\documentclass[reqno]{amsart}%
\usepackage{graphicx}
\usepackage{amsmath}%
\usepackage{amsfonts}%
\usepackage{amssymb}
\newtheorem{theorem}{Theorem}[section]

\newtheorem{example}[theorem]{Example}

\newtheorem{lemma}[theorem]{Lemma}

\newtheorem{proposition}[theorem]{Proposition}
\newtheorem{remark}[theorem]{Remark}

\begin{document}
\title{Hopf Galois Structures on Primitive Purely Inseparable Extensions}
\author{Alan Koch}
\address{Department of Mathematics, Agnes Scott College, 141 E. College Ave., Decatur,
GA\ 30030 USA}
\email{akoch@agnesscott.edu}
\date{\today       }

\begin{abstract}
Let $L/K$ be a primitive purely inseparable extension of fields of
characteristic $p$, $\left[  L:K\right]  >p.$ It is well known that $L/K$ is
Hopf Galois for some Hopf algebra $H$, and it is suspected that $L/K$ is Hopf
Galois for numerous choices of $H$. We construct a family of $K$-Hopf algebras
$H$ for which $L$ is an $H$-Galois object. For some choices of $K$ we will
exhibit an infinite number of such $H.$ We provide some explicit examples of
the dual, Hopf Galois, structure on $L/K.$

\end{abstract}
\maketitle

Let $L/K$ be a modular extension. In \cite{Chase76}, Chase shows that $L$ is a
principal homogenous space for some infinitesimal $K$-group scheme $G$. If
$G=\operatorname*{Spec}H$ then $H$ is a finite dimensional, commutative,
cocommutative $K$-Hopf algebra which is local with local linear dual
(hereafter, ``local-local''), and $L$ is an $H$-Galois object. Interpreted
using duality, this shows that $L/K$ is Hopf Galois for some finite
dimensional, commutative, cocommutative, local-local $K$-Hopf algebra.

A natural question arises: for a given extension $L/K,$ is it Hopf Galois for
a unique choice of $H$? It is well-known that the answer to this question is
``no''. Modular extensions are, by their definition, purely inseparable. In
the work cited above, Chase writes ``[s]crutiny of the simplest examples shows
that a modular extension can be a PHS for many different truncated [$K$]-group
schemes $G.$'' This comment suggests that $L$ is an $H$-Galois object for many
choices of $H.$

The question is then modified: for a given extension $L/K,$ can we describe
all of the Hopf algebras which make it Hopf Galois? The separable analogue,
where $L/K$ is a separable extension, is definitively answered in
\cite{GreitherPareigis87}, which describes all such $H$ using group theory.
The Hopf algebras in the separable case correspond to a certain class of
normal, regular subgroups of the symmetric group $S_{n}$ where $n=\left[
L:K\right]  .$ This elegant result shows that the classification of Hopf
Galois structures depends not on the fields but on the group. Clearly the
number of such Hopf algebras is finite.

In this work, we focus primarily on the simplest class of modular extensions,
namely the primitive extensions. It is well-known that if $L$ is an $H$-Galois
object then $H$ must be monogenic, that is, $H$ is generated as a $K$-algebra
by a single element. We will construct a family of monogenic $K$-Hopf algebras
of dimension equal to $\left[  L:K\right]  =p^{n},$ $n\geq2,$ and show that
each makes $L$ into a Hopf Galois object. These Hopf algebras fall into $n-1$
classes, and the $r^{\text{th}}$ class is parameterized by elements of
$K^{\times}/\left(  K^{p^{r+1}-1}\right)  ^{\times}.$ Not only is $L$ an
$H$-Galois object for each of our constructed Hopf algebras, the realization
of $L$ as an $H$-Galois object can be done in $p^{n-1}\left(  p-1\right)  $
different ways. Unlike the separable case, the number of Hopf Galois
extensions evidently depends on the fields; in particular, our work will
produce examples where the extension $L/K$ is Hopf Galois for an infinite
number of Hopf algebras.

We will also briefly discuss general modular extensions.\ The work presented
in the simple case extends to modular extensions quite easily, however we will
show that there are modular extensions which are $H$-Galois objects for Hopf
algebras which cannot be constructed in the manner presented here.

It should be noted that these constructions can also be done geometrically,
using the language of group schemes and principal homogenous spaces (or
``torsors''). Indeed, the Hopf algebras we construct represent certain
subgroups of group schemes of finite length Witt vectors. We have opted to
present our results using a purely algebraic approach for three reasons.
First, the language in \cite{GreitherPareigis87} is one of Hopf algebras and
Hopf-Galois extensions, and as we are investigating an inseparable analogue to
the results in that paper it seems natural to try to use the same language as
much as possible. Second, we feel the question is stated more naturally using
Hopf algebras -- ``given a field extension $L/K$, for which Hopf algebras is
it a Hopf Galois extension'' makes the point more directly than ``for
$\operatorname*{Spec}K\rightarrow\operatorname*{Spec}L,$ $L$ a field, for
which group schemes $G$ does $\operatorname*{Spec}L$ appear as a torsor?''
does. Third, in \cite{Koch14b} we use these Hopf algebras to describe the ring
of integers in the case where $L$ and $K$ are local fields; an explicit
description of the Hopf algebra action is necessary in that work.

Throughout, $p$ is a fixed prime and $K$ is an imperfect field containing a
perfect field $k$ of characteristic $p.$ All unadorned tensors are over $K$.
We denote by $K^{p^{-\infty}}$ the perfect closure of $K$ in some algebraic
closure. All rings (and algebras) are assumed to be commutative. All Hopf
algebras are assumed to be finite, commutative, cocommutative, of $p$-power
rank, and local-local.

The author would like to thank Nigel Byott and Lindsay Childs for their input
during the creation of this paper.

\section{Background}

We briefly describe the notion of Hopf Galois extensions and Hopf Galois
objects. More details can be found, e.g., in \cite{Childs00}. Let $H$ be a
$K$-Hopf algebra, comultiplication $\Delta$ and counit $\varepsilon.$ We say
that $L$ is an $H$-module algebra if $L$ is an $H$-module such that, for all
$a,b\in L$ we have%
\[
h\left(  ab\right)  =\operatorname{mult}\Delta\left(  h\right)  \left(
a\otimes b\right)  .
\]
If furthermore the $K$-linear map $L\otimes H\rightarrow\operatorname*{End}%
_{K}\left(  L\right)  ,\;\left(  a\otimes h\right)  \mapsto\left(  b\mapsto
ah\left(  b\right)  \right)  $ is an isomorphism, then we say $L/K$ is an
$H$-Galois extension, or simply Hopf Galois if the Hopf algebra is understood.
This can be seen as a generalization of the usual Galois theory: if $E/F$ is a
Galois extension, $\Gamma=\operatorname*{Gal}\left(  E/F\right)  $ then $E/F$
is Hopf Galois via the group algebra $F\left[  \Gamma\right]  .$

Loosely, the notion of a Hopf Galois object is dual to that of a Hopf Galois
extension. Given a $K$-Hopf algebra $H$, suppose there is a map $\alpha
:L\rightarrow L\otimes H$ such that
\begin{align*}
\left(  \alpha\otimes1\right)  \alpha &  =\left(  1\otimes\Delta\right)
\alpha:L\rightarrow L\otimes H\otimes H\\
\operatorname{mult}\left(  1\otimes\varepsilon\right)  \alpha &
=\text{id}_{L}.
\end{align*}
Then $L$ is said to be an $H$-comodule. If furthermore the map $\gamma
:L\otimes L\rightarrow L\otimes H$ given by $\gamma\left(  a\otimes b\right)
=\left(  a\otimes1\right)  \alpha\left(  b\right)  $ is an isomorphism, then
$L$ is an $H$-Galois object (or $H$-principal homogeneous space). It can be
shown that $L/K$ is $H$-Galois if and only if $L$ is an $H^{\ast}$-Galois
object, where $H^{\ast}=\operatorname*{Hom}_{K}\left(  H,K\right)  $ is the
linear dual to $H$.

In the sections that follow, we construct Hopf algebras $H$ for which $L$ is
an $H$-Galois object. Dualizing will put $H^{\ast}$-Galois structures on $L/K.$

The construction of both the comultiplication maps $\Delta$ and coaction maps
$\alpha$ that follow rely heavily on Witt vector addition. Much of the
background on Witt vectors can be found in \cite{Jacobson75}. For the
convenience of the reader we will briefly recall the construction and
illustrate some its properties which will be necessary for the rest of the paper.

For each positive integer $d$, define
\[
w_{d}\left(  Z_{0},\dots,Z_{d}\right)  =Z_{0}^{p^{d}}+pZ_{1}^{p^{d-1}}%
+\cdots+p^{d}Z_{d}\in\mathbb{Z}\left[  Z_{0},\dots,Z_{d}\right]  .
\]
The polynomials $w_{d}$ are called Witt polynomials. Define $S_{d}%
:=S_{d}\left(  \left(  X_{0},\dots,X_{d}\right)  ;\left(  Y_{0},\dots
,Y_{d}\right)  \right)  $ recursively by%
\[
w_{d}\left(  S_{0},\dots,S_{d}\right)  =w_{d}\left(  X_{0},\dots,X_{d}\right)
+w_{d}\left(  Y_{0},\dots,Y_{d}\right)  ,
\]
i.e.,%
\[
S_{d}=\frac{1}{p^{d}}\left(  w_{d}\left(  X_{0},\dots,X_{d}\right)
+w_{d}\left(  Y_{0},\dots,Y_{d}\right)  -S_{0}^{p^{d}}-S_{1}^{p^{d-1}}%
-\cdots-S_{d-1}^{p}\right)  .
\]
Clearly $S_{d}\in\mathbb{Q}\left[  X_{0},\dots,X_{d},Y_{0},\dots,Y_{d}\right]
$; a less obvious, but fundamental, result is that we in fact have $S_{d}%
\in\mathbb{Z}\left[  X_{0},\dots,X_{d},Y_{0},\dots,Y_{d}\right]  .$ To give
two explicit examples of these polynomials,%
\begin{align*}
S_{0}\left(  X_{0};Y_{0}\right)   &  =X_{0}+Y_{0}\\
S_{1}\left(  \left(  X_{0},X_{1}\right)  ;\left(  Y_{0},Y_{1}\right)
\right)   &  =X_{1}+Y_{1}-\sum_{i=1}^{p-1}\frac{\left(  p-1\right)
!}{i!\left(  p-i\right)  !}X_{0}^{i}Y_{0}^{p-i}.
\end{align*}
Let $W\left(  \mathbb{Z}\right)  =\left\{  \left(  a_{0},a_{1},\dots\right)
:a_{i}\in\mathbb{Z}\right\}  ,$ and define a binary operation on $W\left(
\mathbb{Z}\right)  $ by%
\[
\left(  a_{0},a_{1},\dots\right)  +\left(  b_{0},b_{1},\dots\right)  =\left(
S_{0}\left(  a_{0};b_{0}\right)  ,S_{1}\left(  \left(  a_{0},a_{1}\right)
;\left(  b_{0},b_{1}\right)  \right)  ,\dots\right)  .
\]
Then $W\left(  \mathbb{Z}\right)  $ is a group: the identity is $\left(
0,0,\dots\right)  $ and the additive inverse is obtained by negating the
components. This is typically proved by observing that the map%
\begin{align*}
W\left(  \mathbb{Z}\right)   &  \rightarrow\prod_{i=0}^{\infty}\mathbb{Z}\\
\left(  a_{0},a_{1},\dots\right)   &  \rightarrow\left(  w_{0}\left(
a_{0}\right)  ,w_{1}\left(  a_{0},a_{1}\right)  ,\dots\right)
\end{align*}
is a bijection from which $W\left(  \mathbb{Z}\right)  $ inherits its
structure from the product on the right. As $W\left(  \mathbb{Z}\right)  $ is
a group, the component operation $S_{d}$ is associative for all $d$.

By replacing $\mathbb{Z}$ with a $\mathbb{Z}$-algebra $R,$ we obtain the group
$W\left(  R\right)  .$ The polynomials $S_{d}$ can then be viewed as elements
of $R\left[  X_{0},\dots,X_{d},Y_{0},\dots,Y_{d}\right]  .$ In fact, for any
commutative ring $R$ we may view $W$ as an $R$-group scheme.

We conclude this section by recording two well-known observations which will
be needed later. The first equality holds because $S_{d}$ is a polynomial
expression in $x_{0},\dots,x_{d},y_{0},\dots,y_{d}.$ The second follows from
that fact that $S_{d}\left(  \left(  X_{0},\dots X_{d}\right)  ;\left(
Y_{0},\dots Y_{d}\right)  \right)  $ is a homogenous polynomial of degree
$p^{d},$ where $X_{i}$ and $Y_{i}$ each have weight $p^{i}$, $0\leq i\leq d.$

\begin{lemma}
\label{wittlem}Let $A$ and $B$ be $K$-algebras. Let $f:A\rightarrow B$ be a
$K$-algebra map. Then, for $\left(  x_{0},x_{1}\dots\right)  ,\left(
y_{0},y_{1},\dots\right)  \in W\left(  A\right)  ,\;c\in K$ we have, for all
$d$,
\begin{align*}
f\left(  S_{d}\left(  \left(  x_{0},\dots,x_{d}\right)  ;\left(  y_{0}%
,\dots,y_{d}\right)  \right)  \right)   &  =S_{d}\left(  \left(  f\left(
x_{0}\right)  ,\dots,f\left(  x_{d}\right)  \right)  ;\left(  f\left(
y_{0}\right)  ,\dots,f\left(  y_{d}\right)  \right)  \right)  \in B\\
cS_{d}\left(  \left(  x_{0},\dots,x_{d}\right)  ;\left(  y_{0},\dots
,y_{d}\right)  \right)   &  =S_{d}\left(  \left(  c^{p^{-d}}x_{0},\dots
,cx_{d}\right)  ;\left(  c^{p^{-d}}y_{0},\dots,cy_{d}\right)  \right)  \in A
\end{align*}
\end{lemma}

\begin{remark}
Since $K$ is not perfect, it is possible that $c^{p^{-i}}\notin K.$ However,
we may view these as elements of $K^{p^{-\infty}}.$
\end{remark}

\section{Monogenic Hopf Algebras}

The objective of this section is to introduce a new family of monogenic
$K$-Hopf algebras. We will accomplish this by generalizing a classification of
monogenic $k$-Hopf algebras (recall $k$ is perfect). We do not claim that our
adaptation to $K$ yields all monogenic Hopf algebras.

First, we briefly describe the collection of monogenic Hopf algebras over $k$.
This classification appears in Dieudonn\'{e} module form in \cite{Koch01b} and
explicit Hopf algebra descriptions are given in \cite{Koch03}.

Fix a positive integer $n$. By \cite[14.4]{Waterhouse79}, all monogenic Hopf
algebras of rank $p^{n}$ share the same $k$-algebra structure, namely
$H=k\left[  t\right]  /\left(  t^{p^{n}}\right)  ,$ so a study of Hopf algebra
structures reduces to studying the various comultiplications one can put on
this $k$-algebra. The simplest comultiplication can be obtained by letting
$\Delta\left(  t\right)  =t\otimes_{k}1+1\otimes_{k}t$ -- that is, $t$ is a
primitive element. The others are best described using Witt vector addition polynomials.

Let $0<r<n.$ For $\eta\in k^{\times},$ define a sequence $\left\{  \eta
_{i}:i\in\mathbb{Z}^{+}\right\}  $ recursively by%
\begin{align*}
\eta_{1}  &  =\eta\\
\eta_{i}  &  =\eta^{p^{1-i}}\eta_{i-1}^{p^{r}}.
\end{align*}
Notice that the $\eta_{i}$ implicitly depend on $r$. Explicitly, we have
$\eta_{i}=\eta^{e_{i}},$ where%
\begin{equation}
e_{i}=p^{-\left(  i-1\right)  }+p^{r-\left(  i-2\right)  }+p^{2r-\left(
i-3\right)  }+\cdots+p^{\left(  i-1\right)  r}=\sum_{j=0}^{i-1}p^{jr-\left(
i-j-1\right)  }. \label{exp}%
\end{equation}

Let $H=k\left[  t\right]  /\left(  t^{p^{n}}\right)  $, and let $d=\left\lceil
n/r\right\rceil -1.$ Define $\Delta:H\rightarrow H\otimes_{k}H$ by%
\[
\Delta\left(  t\right)  =S_{d}\left(  \left(  \eta_{d}t^{p^{dr}}\otimes
_{k}1,\dots,\eta_{1}t^{p^{r}}\otimes_{k}1,t\otimes_{k}1\right)  ;\left(
1\otimes_{k}\eta_{d}t^{p^{dr}},\dots,1\otimes_{k}\eta_{1}t^{p^{r}}%
,1\otimes_{k}t\right)  \right)  .
\]
\newline This gives $H\;$the structure of a $k$-Hopf algebra with counit
$\varepsilon\left(  t\right)  =0$ and antipode $\lambda\left(  t\right)  =-t.$
We will denote this Hopf algebra by $H_{n,r,\eta}.\;$To see that the Hopf
algebra axioms are satisfied, first notice that coassociativity follows from
the associativity of $S_{d}.$ Also we use Lemma \ref{wittlem} to obtain%
\begin{multline}
\operatorname{mult}\left(  1\otimes\varepsilon\right)  \Delta\left(  t\right)
=S_{d}\left(  \left(  \eta_{d}t^{p^{dr}}\varepsilon\left(  1\right)
,\dots,t\varepsilon\left(  1\right)  \right)  ;\left(  \eta_{d}\varepsilon
\left(  t\right)  ^{p^{dr}},\dots,\varepsilon\left(  t\right)  \right)
\right)  =0\label{witt}\\
\operatorname{mult}\left(  1\otimes\lambda\right)  \Delta\left(  t\right)
=S_{d}\left(  \left(  \eta_{d}t^{p^{dr}}\lambda\left(  1\right)
,\dots,t\lambda\left(  1\right)  \right)  ;\left(  \eta_{d}\lambda\left(
t\right)  ^{p^{dr}},\dots,\lambda\left(  t\right)  \right)  \right)  =0.
\end{multline}

It is shown in \cite{Koch01b} that all of the monogenic local-local Hopf
algebras of dimension $p^{n}$ are of the form $H_{n,r,\eta}.$ Furthermore,
$H_{n,r^{\prime},\eta^{\prime}}\cong H_{n,r,\eta}$ if and only if
$r=r^{\prime}$ and $\eta^{\prime}/\eta=\beta^{p^{r}-p^{-1}}$ for some
$\beta\in k.$ In the case where $k$ is finite, this allows us to count the
number of monogenic Hopf algebras \cite[Cor. 3.2]{Koch01b}. On the other hand,
if $k$ is algebraically closed, there are exactly $n$ monogenic Hopf algebras
of rank $p^{n}.$

Now, we adapt the classification in the perfect field case to the case where
$K$ contains $k$ and is imperfect. Certainly, $H_{n,r,\eta}\otimes_{k}K$ is a
$K$-Hopf algebra of dimension $p^{n}$. However, a careful reading of the
results above reveals that $\eta$ can be replaced by a more general element of
$K$.

Pick $0<r<n$ . Let $g\in\left(  K^{p^{-\infty}}\right)  ^{\times}$ and define
\[
g_{1}=g,\;g_{i}=g^{p^{1-i}}g_{i-1}^{p^{r}},\;i>1.
\]
Note that $g_{i}\notin K$ in general, even if $g\in K.$ However, it can easily
be shown that if $g^{p}\in K$ then $g_{i}^{p^{i}}\in K$.

Our strategy will be to construct a comultiplication $\Delta$ on $H=K\left[
t\right]  /\left(  t^{p^{n}}\right)  \;$such that
\[
\Delta\left(  t\right)  =S_{d}\left(  \left(  g_{d}t^{p^{dr}}\otimes
1,\dots,g_{1}t^{p^{r}}\otimes1,t\otimes1\right)  ;\left(  1\otimes
g_{d}t^{p^{dr}},\dots,1\otimes g_{1}t^{p^{r}},1\otimes t\right)  \right)
\]
for $g\in K^{1/p}$ (although the final form will differ slightly from this).
In order to do so, we need to prove that the expression above is an element of
$H\otimes H.$ The following result accomplishes this.

\begin{lemma}
Let $g\in K^{1/p},$ and let $\left\{  g_{i}\right\}  $ be defined as above.
Then for all $d\geq0,$%
\[
S_{d}\left(  \left(  g_{d}u^{p^{dr}},\dots,g_{1}u^{p^{r}},u\right)  ;\left(
g_{d}v^{p^{dr}},\dots,g_{1}v^{p^{r}},v\right)  \right)  \in K\left[
u,v\right]  .
\]
\end{lemma}

\begin{proof}
In a manner similar to eq. $\left(  \ref{exp}\right)  $ we have
\[
g_{i}=g^{e_{i}},e_{i}=\sum_{j=0}^{i}p^{jr-\left(  i-j-1\right)  }=p^{-i+1}%
\sum_{j=0}^{i}p^{j\left(  r+1\right)  }%
\]
for all $0\leq i\leq d.$ Thus,%
\[
g_{i}=g^{p^{-i+1}}g^{\sum_{j=0}^{i}p^{j\left(  r+1\right)  }},
\]
and the second factor, which we will denote by $g_{i}^{\prime},$ is in $K.$
Then$\;g_{i}=\left(  g^{p}\right)  ^{p^{-i}}g_{i}^{\prime},$ and by Lemma
\ref{wittlem}, we can factor $g^{p}$ out of $S_{d}\left(  \left(
g_{d}u^{p^{dr}},\dots,g_{1}u^{p^{r}},u\right)  ;\left(  g_{d}v^{p^{dr}}%
,\dots,g_{1}v^{p^{r}},v\right)  \right)  $ and obtain%
\begin{align*}
S_{d} &  \left(  \left(  g_{d}u^{p^{dr}},\dots,g_{1}u^{p^{r}},u\right)
;\left(  g_{d}v^{p^{dr}},\dots,g_{1}v^{p^{r}},v\right)  \right)  \\
&  =\tilde{S}_{d}\left(  \left(  g_{d}u^{p^{dr}},\dots,g_{1}u^{p^{r}}\right)
;\left(  g_{d}v^{p^{dr}},\dots,g_{1}v^{p^{r}}\right)  \right)  +u+v\\
&  =g^{p}\tilde{S}_{d}\left(  \left(  g_{d}^{\prime}u^{p^{dr}},\dots
,g_{1}^{\prime}u^{p^{r}}\right)  ;\left(  g_{d}^{\prime}v^{p^{dr}},\dots
,g_{1}^{\prime}v^{p^{r}}\right)  \right)  +u+v,
\end{align*}
where $\tilde{S}_{d}\left(  \left(  X_{0},\dots,X_{d-1}\right)  ;\left(
Y_{0},\dots,Y_{d-1}\right)  \right)  $ is a homogenous polynomial (where
$X_{i}$ and $Y_{i}$ each have weight $p^{i}$, as before). Thus, since
$g^{p}\in K$ and $g_{i}^{\prime}\in K$ for all $i$,
\[
S_{d}\left(  \left(  g_{d}u^{p^{dr}},\dots,g_{1}u^{p^{r}},u\right)  ;\left(
g_{d}v^{p^{dr}},\dots,g_{1}v^{p^{r}},v\right)  \right)  \in K\left[
u,v\right]  .
\]
\end{proof}

By picking $g\in K^{1/p}$ we get a well-defined algebra map on $H$ using the
lemma above with $u=t\otimes1,\;v=1\otimes t.\ $However, we obtain a nicer
parameterization of these maps by letting $f=g^{p}\in K^{\times}.$ Let
$f_{1}=f^{1/p},\;f_{i}=f_{1}^{p^{1-i}}f_{i-1}^{p^{r}}=f^{p^{-i}}f_{i-1}%
^{p^{r}}.$

\begin{proposition}
\label{class}Let $0<r<n$ be integers. Let $d=\left\lceil n/r\right\rceil -1.$
Let $f\in K^{\times}.\;$Let $f_{1},f_{2},\dots,f_{d}$ be the sequence given
recursively by
\[
f_{1}=f^{1/p},\;f_{i}=f^{p^{-i}}f_{i-1}^{p^{r}},\;i\geq2
\]
as above. Let $H_{n,r,f}$ be the $K$-algebra $K\left[  t\right]  /\left(
t^{p^{n}}\right)  ,$ and let%
\begin{align*}
\Delta\left(  t\right)   &  =S_{d}\left(  \left(  f_{d}t^{p^{dr}}%
\otimes1,\dots,f_{1}t^{p^{r}}\otimes1,t\otimes1\right)  ;\left(  1\otimes
f_{d}t^{p^{dr}},\dots,1\otimes f_{1}t^{p^{r}},1\otimes t\right)  \right) \\
\varepsilon\left(  t\right)   &  =0\\
\lambda\left(  t\right)   &  =-t
\end{align*}
Then these maps endow $H_{n,r,f}$ with the structure of a $K$-Hopf algebra.
\end{proposition}

\begin{proof}
Since the comultiplication is accomplished using Witt vector sums, the
computations here are the same as in eq. $\left(  \ref{witt}\right)  $.
Alternatively, we could use the facts that $H_{n,r,f}\otimes K^{p^{-\infty}}$
is a Hopf algebra by \cite{Koch01b}, and $\Delta\left(  t\right)
,\varepsilon\left(  t\right)  ,$ and $\lambda\left(  t\right)  $ are defined
over $K.$
\end{proof}

\section{Isomorphism Questions}

The Hopf algebras $\left\{  H_{n,r,f}:0<r<n,\;f\in K^{\times}\right\}  $
constructed in Proposition \ref{class} are not all unique. While $n$ and $r$
are isomorphism invariants, different choices of $f$ can lead to isomorphic
Hopf algebras. Here, we will give a sufficient condition on $f,f^{\prime}$ for
$H_{n,r,f}\cong H_{n,r,f^{\prime}}.$ Additionally, if $r$ is sufficiently
large (with respect to $n$) then we will see this condition is necessary as well.

Suppose $H_{n,r,f}=K\left[  t\right]  /\left(  t^{p^{n}}\right)  .$ Pick $g\in
K^{\times},$ and let $u=gt.$ Then, as a $K$-algebra, $H_{n,r,f}=K\left[
u\right]  /\left(  u^{p^{n}}\right)  ;$ with the help of Lemma \ref{wittlem}
we have%
\begin{align*}
\Delta\left(  u\right)   &  =g\Delta\left(  t\right) \\
&  =gS_{d}\left(  \left(  f_{d}t^{p^{dr}}\otimes1,\dots,f_{1}t^{p^{r}}%
\otimes1,t\otimes1\right)  ;\left(  1\otimes f_{d}t^{p^{dr}},\dots,1\otimes
f_{1}t^{p^{r}},1\otimes t\right)  \right) \\
&  =S_{d}\left(  \left(  g^{p^{-d}}f_{d}t^{p^{dr}}\otimes1,\dots,g^{p^{-1}%
}f_{1}t^{p^{r}}\otimes1,gt\otimes1\right)  ;\left(  1\otimes g^{p^{-d}}%
f_{d}t^{p^{dr}},\dots,1\otimes g^{p^{-1}}f_{1}t^{p^{r}},1\otimes gt\right)
\right) \\
&  =S_{d}(\left(  g^{p^{-d}-p^{dr}}f_{d}\left(  gt\right)  ^{p^{dr}}%
\otimes1,\dots,g^{p^{-1}-p^{r}}f_{1}\left(  gt\right)  ^{p^{r}}\otimes
1,gt\otimes1\right)  ;\\
&  \left(  1\otimes g^{p^{-d}-p^{dr}}f_{d}\left(  gt\right)  ^{p^{dr}}%
,\dots,g^{p^{-1}-p^{r}}f_{1}\left(  gt\right)  ^{p^{r}}t^{p^{r}},1\otimes
gt\right)  )\\
&  =S_{d}\left(  \left(  g_{d}f_{d}u^{p^{dr}}\otimes1,\dots,g_{1}f_{1}%
u^{p^{r}}\otimes1,u\otimes1\right)  ;\left(  1\otimes g_{d}f_{d}u^{p^{dr}%
},\dots,g_{1}f_{1}u^{p^{r}},1\otimes u\right)  \right)  ,
\end{align*}
where $g_{i}=g^{p^{-i}-p^{ir}}.$ Now%
\begin{align*}
g_{1}^{p^{1-i}}g_{i-1}^{p^{r}}  &  =\left(  g^{p^{-1}-p^{r}}\right)
^{p^{1-i}}\left(  g^{p^{-\left(  i-1\right)  }-p^{\left(  i-1\right)  r}%
}\right)  ^{p^{r}}\\
&  =g^{p^{-i}-p^{r+1-i}}g^{p^{r+1-i}-p^{ir-r+r}}\\
&  =g^{p^{-i}-p^{ir}}\\
&  =g_{i},
\end{align*}
and
\[
g_{1}^{p}=\left(  g^{p^{-1}-p^{r}}\right)  ^{p}=g^{1-p^{r+1}}.
\]
From this it follows that $H_{n,r,f}\cong H_{n,r,\left(  g^{1-p^{r+1}}\right)
f}.$ More generally,

\begin{proposition}
\label{isoprop}Let $f,f^{\prime}\in K^{\times}.$ If $f/f^{\prime}\in\left(
K^{\times}\right)  ^{p^{r+1}-1}$ then $H_{n,r,f}\cong H_{n,r,f^{\prime}}.$
\end{proposition}

\begin{proof}
The statement that $H_{n,r,f}\cong H_{n,r,f^{\prime}}$ whenever $f/f^{\prime
}\in\left(  K^{\times}\right)  ^{p^{r+1}-1}$ has been proved already.
Conversely, suppose $H_{n,r,f}\cong H_{n,r,f^{\prime}}.$ Then $H_{n,r,f}%
\otimes K^{p^{-\infty}}\cong H_{n,r,f^{\prime}}\otimes K^{p^{-\infty}},$ of
course, and, by \cite[Sec. 3]{Koch01b}, $f/f^{\prime}\in\left(  \left(
K^{p^{-\infty}}\right)  ^{\times}\right)  ^{p^{r+1}-1}.$ Thus, the equation%
\[
x^{p^{r+1}}+\frac{f}{f^{\prime}}x=0
\]
has a solution in $K^{p^{-\infty}}.$ If $g$ is such a solution, then $K\left(
g\right)  $ is a separable extension of $K$ contained in $K^{p^{-\infty}}.$
Since $K^{p^{-\infty}}/K$ is purely inseparable we have $g\in K,$ hence
$f/f^{\prime}\in\left(  K^{\times}\right)  ^{p^{r+1}-1}.$
\end{proof}

Note the parallel with the result in \cite{Koch01b} if we replace $\beta$ with
$\beta^{p^{-1}}$-- of course, if $f$ and $f^{\prime}$ are constants these Hopf
algebras are defined over $k$ and we expect the isomorphism condition above to hold.

\section{Hopf Galois Objects}

Let $L=K\left(  x\right)  ,\;x^{p^{n}}=b\in K.$ If $H$ is the monogenic Hopf
algebra with primitive generator, then $L$ is an $H$-Galois object: this is
the Hopf algebra used in the construction of \cite{Chase76}, where he shows
that all modular extensions of $K$ are Hopf Galois objects. The purpose of
this section is to show that each of the rank $p^{n}$ Hopf algebras
constructed above can be used to make $L$ a Hopf-Galois object.

Let $H=H_{n,r,f}$ for some choice of $0<r<n$ and $f\in K.$ Define
$\alpha:L\rightarrow L\otimes H$ by%
\[
\alpha\left(  x\right)  =S_{d}\left(  \left(  f_{d}x^{p^{dr}}\otimes
1,\dots,f_{1}x^{p^{r}}\otimes1,x\otimes1\right)  ;\left(  1\otimes
f_{d}t^{p^{dr}},\dots,1\otimes f_{1}t^{p^{r}},1\otimes t\right)  \right)  .
\]
Since exponentiation-by-$p$ is a $K$-algebra map,
\begin{align*}
\alpha\left(  b\right)   &  =\alpha\left(  x^{p^{n}}\right) \\
&  =S_{d}\left(  \left(  f_{d}^{p^{n}}\left(  x^{p^{dr}}\right)  ^{p^{n}%
}\otimes1,\dots,x^{p^{n}}\otimes1\right)  ;\left(  1\otimes f_{d}^{p^{n}%
}\left(  t^{p^{dr}}\right)  ^{p^{n}},\dots,1\otimes f_{1}\left(  t^{p^{r}%
}\right)  ^{p^{n}},1\otimes t^{p^{n}}\right)  \right) \\
&  =S_{d}\left(  \left(  f_{d}^{p^{n}}\left(  x^{p^{dr}}\right)  ^{p^{n}%
}\otimes1,\dots,x^{p^{n}}\otimes1\right)  ;\left(  0,\dots,0\right)  \right)
=x^{p^{n}}\otimes1=b\otimes1,
\end{align*}
and so $\alpha$ is a well-defined $K$-algebra map.

\begin{lemma}
The map $\alpha$ above gives $L$ the structure of a right $H$-comodule.
\end{lemma}

\begin{proof}
We need to show that $\left(  1\otimes\Delta\right)  \alpha\left(  x\right)
=\left(  \alpha\otimes1\right)  \alpha\left(  x\right)  $ and $\mu\left(
1\otimes\varepsilon\right)  \alpha\left(  x\right)  =x.$ The first follows
immediately from the associativity of $S_{d}.$ The second computation is
similar to the one in eq. $\left(  \ref{witt}\right)  $.
\end{proof}

\begin{proposition}
Let $\gamma:L\otimes L\rightarrow L\otimes H$ be given by
\[
\gamma\left(  a\otimes b\right)  =\left(  a\otimes1\right)  \alpha\left(
b\right)  .
\]
Then $\gamma$ is a $K$-module isomorphism, hence $L$ is an $H$-Galois object.
\end{proposition}

\begin{proof}
First, notice that since $\alpha$ is a $K$-algebra map we have that $\gamma$
preserves multiplication, i.e., $\gamma\left(  \left(  a\otimes b\right)
\left(  c\otimes d\right)  \right)  =\gamma\left(  a\otimes b\right)
\gamma\left(  c\otimes d\right)  $. Also, $\gamma$ is an $L$-module map if we
view $L\otimes L$ as an $L$-module via the first factor since $\gamma\left(
a\otimes b\right)  =\left(  a\otimes1\right)  \gamma\left(  b\right)  .$ Since
$L\otimes L$ and $L\otimes H$ are both $K$-vector spaces of dimension
$p^{2n},$ it suffices to show that $\gamma$ is onto. Now%
\begin{align*}
\gamma\left(  -x\otimes1+1\otimes x\right)   &  =-\left(  x\otimes1\right)
\alpha\left(  1\right)  +\left(  1\otimes1\right)  \alpha\left(  x\right) \\
&  =-x\otimes1+S_{d}\left(  \left(  f_{d}x^{p^{dr}}\otimes1,\dots
,x\otimes1\right)  ;\left(  1\otimes f_{d}t^{p^{dr}},\dots,1\otimes t\right)
\right) \\
&  =-x\otimes1+x\otimes1+1\otimes t+t^{2}\xi\\
&  =1\otimes t+t^{2}\xi_{1}%
\end{align*}
for some $\xi_{1}\in L\otimes H.$ As $\gamma$ is multiplicative we see that
\[
\gamma\left(  \left(  -x\otimes1+1\otimes x\right)  ^{i}\right)  =1\otimes
t^{i}+t^{i+1}\xi_{i},\;\xi_{i}\in L\otimes H
\]
for $1\leq i<p^{n}.\;$Thus $\left\{  \gamma\left(  \left(  -x\otimes1+1\otimes
x\right)  ^{i}\right)  \right\}  $ is an $L$-linearly independent set, we have
$\dim_{L}\operatorname{Im}\gamma\geq p^{n}.$ As $\left[  L:K\right]  =p^{n}$
we have $\dim_{K}\operatorname{Im}\gamma\geq p^{2n}$ so $\operatorname{Im}%
\gamma=L\otimes H.$
\end{proof}

One will notice many parallels between this theory and the Kummer theory of
formal groups construction in, e.g., \cite[Sec. 39]{Childs00}.\ This is to be
expected since a smooth resolution for $\operatorname*{Spec}H_{n,r,f}$ can be
easily constructed by adapting the resolution in \cite[Sec. 4]{Koch03}.

Of course, there are many different descriptions for the same field extension
$L.$ Indeed, pick $g\in K^{\times}$ and let $y=gx.$ Then $y^{p^{n}}=g^{p^{n}%
}b\in K$ and so $L=K\left(  y\right)  .$ With the coaction above we have%
\begin{align*}
\alpha\left(  y\right)   &  =\alpha\left(  gx\right)  \\
&  =gS_{d}\left(  \left(  f_{d}x^{p^{dr}}\otimes1,\dots,x\otimes1\right)
;\left(  1\otimes f_{d}t^{p^{dr}},\dots,1\otimes t\right)  \right)  \\
&  =S_{d}\left(  \left(  g^{p^{-d}}f_{d}x^{p^{dr}}\otimes1,\dots
,x\otimes1\right)  ;\left(  1\otimes g^{p^{-d}}f_{d}t^{p^{dr}},\dots,1\otimes
t\right)  \right)  \\
&  =S_{d}\left(  \left(  g_{d}f_{d}x^{p^{dr}}\otimes1,\dots,g_{1}f_{1}%
x^{p^{r}}\otimes1,x\otimes1\right)  ;\left(  1\otimes g_{d}f_{d}t^{p^{dr}%
},\dots,g_{1}f_{1}u^{p^{r}}t^{p^{r}},1\otimes t\right)  \right)
\end{align*}
where $g_{i}=g^{p^{-i}-p^{ri}}$ as in the previous section. Thus, since
$g_{1}^{p}=g^{1-p^{r+1}},$ changing the generator of $L$ in this manner
results in the same coaction: $H_{n,r,f}$ acts on $x$ in the same way as
$H_{n,r,fg^{1-p^{r+1}}}$ acts on $y$, and these two Hopf algebras are isomorphic.

On the other hand, let $x_{i}=x^{i},\;1\leq i\leq n-1,\;\gcd\left(
p,i\right)  -1.$ Then $L=K\left(  x\right)  =K\left(  x_{i}\right)  ,$ and
defining
\[
\alpha_{i}\left(  x_{i}\right)  =S_{d}\left(  \left(  f_{d}x_{i}^{p^{dr}%
}\otimes1,\dots,f_{1}x_{i}^{p^{r}}\otimes1,x_{i}\otimes1\right)  ;\left(
1\otimes f_{d}t^{p^{dr}},\dots,1\otimes f_{1}t^{p^{r}},1\otimes t\right)
\right)
\]
allows for a coaction of $H_{n,r,f}$ on $L;$ as $i$ varies each resulting
coaction is different. Thus, there are $\phi\left(  p^{n}\right)
=p^{n-1}\left(  p-1\right)  $ different ways to view $L$ as an $H_{n,r,f}%
$-Galois object.

\begin{remark}
Proposition \ref{isoprop} can be used to provide examples of finite field
extensions $L/K\;$with an infinite number of $K$-Hopf algebras which $L$ is an
$H$-Galois object. For example, let $K=k\left(  T_{1},T_{2},\dots\right)  $
and let $L$ be any purely inseparable extension of degree $p^{2}$ (or
greater). Then $H_{n,r,T_{i}}\not \cong H_{n,r,T_{j}}$ unless $i=j.$
\end{remark}

Both Chase's construction and the Hopf algebras presented here can be
considered under one general theory.\ Indeed, were we to allow $r=n$ and
$f=0,$ then $d=0$ and we recover Chase's Hopf algebra. We have chosen to treat
them as separate cases to simplify the question of isomorphic Hopf algebras --
clearly, the Hopf algebra ``$H_{n,n,f}$'' does not depend at all on $f.$

\section{Explicit Computations:\ the Case $r=n-1$}

We shall now explicitly describe the action of $H:=H_{n,r,f}^{\ast}$ on $L$ in
the case where $r=n-1.$ In this case, $d=1,$ and hence the comultiplication on
$H_{n,r,f}$ is
\[
\Delta\left(  t\right)  =t\otimes1+1\otimes t+f\sum_{\ell=1}^{p-1}%
\frac{1}{\ell!\left(  p-\ell\right)  !}t^{p^{r}\ell}\otimes t^{p^{r}\left(
p-\ell\right)  }.
\]
We view this as a restriction on $r$, not on $n$ -- that is, $L/K$ can be any
extension, but and we only consider the Hopf algebras with $r=n-1.$ This will
provide a family of explicit Hopf Galois actions on any purely inseparable
extension $L/K$ of degree $p^{n},\;n\geq2.$

As a $K$-module, $H$ has a basis $\left\{  z_{0}=1,z_{1},\dots,z_{p^{r}%
}\right\}  $ with $z_{i}:H\rightarrow K$ given by
\[
z_{j}\left(  t^{i}\right)  =\delta_{i,j},
\]
where $\delta_{i,j}$ is the Kronecker delta. The algebra structure on $H$ is
induced from the coalgebra structure on $H_{n,r,f};$ explicitly,%
\begin{equation}
z_{j_{1}}z_{j_{2}}\left(  h\right)  =\operatorname{mult}\left(  z_{j_{1}%
}\otimes z_{j_{2}}\right)  \Delta\left(  h\right)  . \label{zmult}%
\end{equation}
We claim that $\left\{  z_{p},z_{p^{2}},z_{p^{3}},\dots,z_{p^{r}}\right\}  $
generate $H$ as a $K$-algebra.

We start with a result which will facilitate the study of the algebra
structure of $H$ as well as the action of $H$ on $L.$

\begin{lemma}
\label{thepowlem}Let%
\[
S_{f}\left(  u,v\right)  =u+v+f\sum_{\ell=1}^{p-1}\frac{1}{\ell!\left(
p-\ell\right)  !}u^{p^{r}\ell}v^{p^{r}\left(  p-\ell\right)  }.
\]
Then, for every positive integer $i$, $S_{f}\left(  u,v\right)  ^{i}$ is a
$K$-linear combination of elements of the form
\[
u^{i_{1}+p^{r}\ell^{\prime}}v^{i_{2}+p^{r}\ell^{\prime\prime}},
\]
where%
\begin{align*}
\ell^{\prime}  &  =i_{3,1}+2i_{3,2}+\cdots+\left(  p-1\right)  i_{3,p-1}\\
\ell^{\prime\prime}  &  =\left(  p-1\right)  i_{3,1}+\left(  p-2\right)
i_{3,2}+\cdots+i_{3,p-1},
\end{align*}
and $i_{3,1}+i_{3,2}+\cdots+i_{3,p-1}=i_{3}.$
\end{lemma}

\begin{proof}
We have%
\begin{align*}
S_{f}\left(  u,v\right)  ^{i}  &  =\left(  u+v+f\sum_{\ell=1}^{p-1}%
\frac{1}{\ell!\left(  p-\ell\right)  !}u^{p^{r}\ell}v^{p^{r}\left(
p-\ell\right)  }\right)  ^{i}\\
&  =\sum_{i_{1}+i_{2}+i_{3}=i}\binom{i}{i_{1},i_{2},i_{3}}\left(  u^{i_{1}%
}v^{i_{2}}\right)  \left(  f\sum_{\ell=1}^{p-1}\frac{1}{\ell!\left(
p-\ell\right)  !}u^{p^{r}\ell}v^{p^{r}\left(  p-\ell\right)  }\right)
^{i_{3}}.
\end{align*}
The last factor in each summand can be expanded as%
\[
f^{i_{3}}\sum_{i_{3,1}+\cdots+i_{3,p-1}}\left(  \binom{i_{3}}{i_{3,1}%
,\dots,i_{3,p-1}}\left(  \prod_{j=1}^{p-1}\frac{1}{i_{3,j}!\left(
p-i_{3,j}\right)  !}\right)  u^{^{i_{1}+p^{r}\ell^{\prime}}}v^{^{i_{2}%
+p^{r}\ell^{\prime\prime}}}\right)  .
\]

The result follows.
\end{proof}

Next, we consider powers of the $z_{p^{s}}$'s.

\begin{lemma}
For $0\leq s\leq r,\;1\leq m\leq p-1,\;z_{p^{s}}^{m}=m!z_{mp^{s}}.$
\end{lemma}

\begin{proof}
Clearly, this holds for $m=1.$ Suppose $z_{p^{s}}^{m-1}=\left(  m-1\right)
!z_{\left(  m-1\right)  p^{s}}.$ By eq. $\left(  \ref{zmult}\right)  $ we have%
\begin{align*}
z_{p^{s}}^{m}\left(  t^{i}\right)   &  =\operatorname{mult}\left(  z_{p^{s}%
}^{m-1}\otimes z_{p^{s}}\right)  \Delta\left(  t^{i}\right) \\
&  =\operatorname{mult}\left(  z_{p^{s}}^{m-1}\otimes z_{p^{s}}\right)
S_{f}\left(  t\otimes1,1\otimes t\right)  ^{i}\\
&  =\operatorname{mult}\left(  \left(  m-1\right)  !z_{\left(  m-1\right)
p^{s}}\otimes z_{p^{s}}\right)  \left(  t\otimes1+1\otimes t+f\sum_{\ell
=1}^{p-1}\frac{1}{\ell!\left(  p-\ell\right)  !}t^{p^{r}\ell}\otimes
t^{p^{r}\left(  p-\ell\right)  }\right)  ^{i}.
\end{align*}
By Lemma \ref{thepowlem}, the tensors are of the form
\[
t^{i_{1}+p^{r}\ell^{\prime}}\otimes t^{i_{2}+p^{r}\ell^{\prime\prime}},
\]
with $\ell^{\prime},\ell^{\prime\prime}$ as before. Note that $\ell^{\prime
}+\ell^{\prime\prime}=pi_{3}.$ Since $z_{p^{s}}\left(  t^{\ell}\right)
=\delta_{p^{s},\ell}$ and $z_{\left(  m-1\right)  p^{s}}\left(  t^{\ell
}\right)  =\delta_{\left(  m-1\right)  p^{s},i},$ the expression $\left(
z_{\left(  m-1\right)  p^{s}}\otimes z_{p^{s}}\right)  \left(  t^{i_{1}%
+p^{r}\ell^{\prime}}\otimes t^{i_{2}+p^{r}\ell^{\prime\prime}}\right)  $ is
nontrivial only if%
\begin{align*}
\left(  m-1\right)  p^{s}  &  =i_{1}+p^{r}\ell^{\prime}\\
p^{s}  &  =i_{2}+p^{r}\ell^{\prime\prime}.
\end{align*}
If we add the two equations together we get%
\[
mp^{s}=i_{1}+i_{2}+p^{r+1}i_{3}.
\]
From this it is clear that $i_{3}=0,$ which means $\ell^{\prime}=\ell
^{\prime\prime}=0$ as well. Thus $i_{2}=p^{s}$ and $i_{1}=\left(  m-1\right)
p^{s},$ hence $i=mp^{s}$ and with the help of Lucas' Theorem \cite{Fine47} we
get%
\begin{align*}
z_{p^{s}}^{m}\left(  t^{mp}\right)   &  =\left(  m-1\right)  !\binom{mp^{s}%
}{\left(  m-1\right)  p^{s},p^{s},0}\\
&  =\left(  m-1\right)  !\binom{mp^{s}}{p^{s}}\\
&  =\left(  m-1\right)  !m\\
&  =m!
\end{align*}
Therefore, $z_{p^{s}}^{m}=m!z_{mp^{s}}.$
\end{proof}

\begin{lemma}
For $0\leq s\leq r-1,\;z_{p^{s}}^{p}=0.$
\end{lemma}

\begin{proof}
We have%
\[
z_{p^{s}}^{p}\left(  t^{i}\right)  =\left(  p-1\right)  !\operatorname{mult}%
\left(  z_{\left(  p-1\right)  p^{s}}\otimes z_{p^{s}}\right)  \left(
\sum_{i_{1}+i_{2}+i_{3}=i}\binom{i}{i_{1},i_{2},i_{3}}\left(  t^{i_{1}}\otimes
t^{i_{2}}\right)  \left(  f\sum_{\ell=1}^{p-1}\frac{1}{\ell!\left(
p-\ell\right)  !}t^{p^{r}\ell}\otimes t^{p^{r}\left(  p-\ell\right)  }\right)
^{i_{3}}\right)  .
\]
If $\left(  z_{\left(  p-1\right)  p}\otimes z_{p}\right)  \left(
t^{i_{1}+p^{r}\ell^{\prime}}\otimes t^{i_{2}+p^{r}\ell^{\prime\prime}}\right)
$ is nontrivial then%
\begin{align*}
\left(  p-1\right)  p^{s}  &  =i_{1}+p^{r}\ell^{\prime}\\
p^{s}  &  =i_{2}+p^{r}\ell^{\prime\prime}.
\end{align*}
Again, $i_{3}=0,$ so $i_{2}=p^{s}$ and $i_{1}=\left(  p-1\right)  p^{s},$
hence $i=p^{s+1}.$ But then%
\begin{align*}
z_{p^{s}}^{p}\left(  t^{p^{s+1}}\right)   &  =\left(  p-1\right)
!\binom{p^{s+1}}{\left(  p-1\right)  p^{s},p^{s},0}\\
&  =-\binom{p^{s+1}}{p^{s}}\\
&  =0.
\end{align*}
So $z_{p^{s}}^{p}=0.$
\end{proof}

\begin{remark}
While not part of the generating set we are constructing, notice that the
above results show $z_{1}^{m}=m!z_{m}$ and $z_{1}^{p}=0.$
\end{remark}

The behavior is slightly different for $p^{r}.$

\begin{lemma}
We have $z_{p^{r}}^{p}=fz_{1}$ and $z_{p^{r}}^{p^{2}}=0.$
\end{lemma}

\begin{proof}
If the expression $\left(  z_{\left(  p-1\right)  p^{r}}\otimes z_{p^{r}%
}\right)  \left(  t^{i_{1}+p^{r}\ell^{\prime}}\otimes t^{i_{2}+p^{r}%
\ell^{\prime\prime}}\right)  $ in the expansion of $z_{p^{r}}^{p}\left(
t^{i}\right)  $is nontrivial then%
\begin{align*}
\left(  p-1\right)  p^{r}  &  =i_{1}+p^{r}\ell^{\prime}\\
p^{r}  &  =i_{2}+p^{r}\ell^{\prime\prime}.
\end{align*}
If $i_{3}=0$ then $i_{2}=p^{r},$ $i_{1}=\left(  p-1\right)  p^{r}$ and
$i=p^{r+1};$ however, $i<p^{r+1}=p^{n}$ so this cannot occur. Thus
$i_{3}=1,\;\;\ell^{^{\prime}}=p-1,\;\ell^{\prime\prime}=1\;$(both of these can
occur only by setting $i_{3,j}=\delta_{j,p-1}$)$,\;i_{2}=0,$ and $i_{1}=0.$
Hence, $i=1$ and%
\[
z_{p^{r}}^{p}\left(  t\right)  =\left(  p-1\right)  !\binom{1}{0,0,1}%
f\frac{1}{\left(  p-1\right)  !\left(  p-\left(  p-1\right)  \right)  !}=f.
\]
Therefore, $z_{p^{r}}^{p}=fz_{1}.$ That $z_{p^{r}}^{p^{2}}=0$ follows immediately.
\end{proof}

From the results above, we can deduce that $\left\{  z_{p^{s}}:1\leq s\leq
r\right\}  $ generate $H$ as a $K$-algebra.

The coalgebra structure on $H$ is induced from the multiplication on
$H_{n,r,f}$ and is much more straightforward. For all $h\in H$ when we apply
the comultiplication $\Delta$ to it we get a $K$-linear map $H_{n,r,f}\otimes
H_{n,r,f}\rightarrow K$ given by%
\[
\Delta\left(  h\right)  \left(  a\otimes b\right)  =h\left(  ab\right)  .
\]
Thus,%
\[
\Delta\left(  z_{j}\right)  \left(  t^{i_{1}}\otimes t^{i_{2}}\right)
=z_{j}\left(  t^{i_{1}+i_{2}}\right)  =\delta_{j,i_{1}+i_{2}}%
\]
and so%
\[
\Delta\left(  z_{j}\right)  =\sum_{i=0}^{j}z_{j-i}\otimes z_{i}.
\]
Note that this is true for all $j$, not just the powers of $p.$

We summarize.

\begin{proposition}
\label{dual}The Hopf algebra $H$ above is
\begin{align*}
H  &  =K\left[  z_{p},z_{p^{2}},\dots,z_{p^{r}}\right]  /\left(  z_{p}%
^{p},z_{p^{2}}^{p},\dots,z_{p^{r-1}}^{p},z_{p^{r}}^{p^{2}}\right) \\
\Delta\left(  z_{p^{s}}\right)   &  =\sum_{i=0}^{p^{s}}z_{p^{s}-i}\otimes
z_{i}.
\end{align*}
\end{proposition}

Of course, it is possible to write $\Delta\left(  z_{p^{s}}\right)  $ solely
in terms of $z_{p},\dots,z_{p^{s}},z_{p^{r}},$ but that is not needed for our purposes.

We will now describe the Hopf Galois action of $H$ on $L.$ Then the a
$K$-algebra map $\alpha:L\rightarrow L\otimes H_{n,r,f}$ is given by%
\[
\alpha\left(  x\right)  =x\otimes1+1\otimes t+f\sum_{\ell=1}^{p-1}%
\frac{1}{\ell!\left(  p-\ell\right)  !}x^{p^{r}\ell}\otimes t^{p^{r}\left(
p-\ell\right)  }.
\]
This gives $L$ the structure of an $H_{n,r,f}$-comodule -- in fact it makes
$L$ an $H_{n,r,f}$-Galois object. Here, we compute the induced action of $H$
on $L$ which makes $L/K$ an $H$-Galois extension.

Generally, if $A$ is a $K$-Hopf algebra such that $L$ is an $A$-Galois object,
then $A^{\ast}$ acts on $L$ by%
\begin{equation}
h\left(  y\right)  =\operatorname{mult}\left(  1\otimes h\right)
\alpha\left(  y\right)  ,\;h\in A^{\ast},y\in L. \label{act}%
\end{equation}
Here, it suffices to compute $z_{p^{s}}\left(  x^{i}\right)  $ for $1\leq
s\leq r,\;1\leq i\leq p^{n}-1,$ however it will also be useful to compute
$z_{j}\left(  x^{i}\right)  $ for some choices of $j$ which are not powers of
$p$. Notice that we use $z_{j}\left(  -\right)  $ in two different contexts:
one to describe $z_{j}$ as a map $H_{n,r,f}\rightarrow K,$ the other to
describe how $z_{j}$ acts on $L.$

The first result handles the case $i=1.$

\begin{lemma}
We have
\[
z_{0}\left(  x\right)  =x,\;z_{1}\left(  x\right)  =1,\;z_{p^{r}}\left(
x\right)  =-fx^{p^{r}\left(  p-1\right)  }.
\]
For $1\leq j\leq p^{r}-1$, $z_{j}\left(  x\right)  =0.$
\end{lemma}

\begin{proof}
Applying eq. $\left(  \ref{act}\right)  $ to $h=z_{j},$ $y=x$ gives%
\[
z_{j}\left(  x\right)  =xz_{j}\left(  1\right)  +z_{j}\left(  t\right)
+f\sum_{\ell=1}^{p-1}\frac{1}{\ell!\left(  p-\ell\right)  !}x^{p\ell}%
z_{j}\left(  t^{p^{r}\left(  p-\ell\right)  }\right)  ,
\]
from which the result follows.
\end{proof}

The second result handles the cases where $i$ is a nontrivial power of $p.$

\begin{lemma}
\label{powcomp}For $1\leq m\leq r$ we have%
\[
z_{0}\left(  x^{p^{m}}\right)  =x^{p^{m}},\;z_{p^{m}}\left(  x^{p^{m}}\right)
=1.
\]
For all other choices of $j,$ $z_{j}\left(  x^{p^{m}}\right)  =0.$
\end{lemma}

\begin{proof}
The computations are facilitated by observing that $\alpha\left(  x^{p^{m}%
}\right)  =x^{p^{m}}\otimes1+1\otimes t^{p^{m}}.$ Indeed,%
\begin{align*}
\alpha\left(  x^{p^{m}}\right)   &  =\alpha\left(  x\right)  ^{p^{m}}\\
&  =\left(  x\otimes1+1\otimes t+f\sum_{\ell=1}^{p-1}\frac{1}{\ell!\left(
p-\ell\right)  !}x^{p^{r}\ell}\otimes t^{p^{r}\left(  p-\ell\right)  }\right)
^{p^{m}}\\
&  =x^{p^{m}}\otimes1+1\otimes t^{p^{m}}+f\sum_{\ell=1}^{p-1}\frac{1}%
{\ell!\left(  p-\ell\right)  !}x^{p^{r+m}\ell}\otimes t^{p^{r+m}\left(
p-\ell\right)  }\\
&  =x^{p^{m}}\otimes1+1\otimes t^{p^{m}}%
\end{align*}
since $r+m\geq r+1=n.$ Now for $0\leq j\leq p^{n}-1$ we have%
\[
z_{j}\left(  x^{p^{m}}\right)  =x^{p^{m}}z_{j}\left(  1\right)  +z_{j}\left(
t^{p^{m}}\right)  ,
\]
from which the result follows.
\end{proof}

Next, we have

\begin{theorem}
\label{action}Let $H=H_{n,r,f}^{\ast}\ $be as in Proposition \ref{dual}, that
is,%
\begin{align*}
H &  =K\left[  z_{p},z_{p^{2}},\dots,z_{p^{r}}\right]  /\left(  z_{p}%
^{p},z_{p^{2}}^{p},\dots,z_{p^{r-1}}^{p},z_{p^{r}}^{p^{2}}\right)  \\
\Delta\left(  z_{p^{s}}\right)   &  =\sum_{i=0}^{p^{s}}z_{p^{s}-i}\otimes
z_{i}.
\end{align*}
For $0\leq i\leq p^{n}-1$, write%
\[
i=\sum_{\ell=0}^{r}i_{\left(  \ell\right)  }p^{\ell}.
\]
Then, for $0\leq s\leq r-1$ we have%
\[
z_{p^{s}}\left(  x^{i}\right)  =i_{\left(  s\right)  }x^{i-p^{s}}.
\]
Additionally,%
\[
z_{p^{r}}\left(  x^{i}\right)  =i_{\left(  r\right)  }x^{i-p^{r}}%
-ifx^{p^{r}\left(  p-1\right)  +i-1}.
\]
\end{theorem}

\begin{remark}
Note that if $i<p^{s}<p^{r}$ then $z_{p^{s}}\left(  x^{i}\right)  =0,$ and if
$i<p^{r}$ then $z_{p^{r}}\left(  x^{i}\right)  =-ifx^{p^{r}\left(  p-1\right)
+i-1}.$
\end{remark}

\begin{proof}
We have%
\begin{align*}
z_{p^{s}}\left(  x^{i}\right)   &  =\operatorname{mult}\left(  1\otimes
z_{p^{s}}\right)  \alpha\left(  x^{i}\right) \\
&  =\operatorname{mult}\left(  1\otimes z_{p^{s}}\right)  S_{f}\left(
x\otimes1,1\otimes t\right)  ^{i}\\
&  =\operatorname{mult}\left(  1\otimes z_{p^{s}}\right)  \left(
x\otimes1+1\otimes t+f\sum_{\ell=1}^{p-1}\frac{1}{\ell!\left(  p-\ell\right)
!}x^{p^{r}\ell}\otimes t^{p^{r}\left(  p-\ell\right)  }\right)  ^{i}\\
&  =\operatorname{mult}\left(  1\otimes z_{p^{s}}\right)  \sum_{i_{1}%
+i_{2}+i_{3}=i}\binom{i}{i_{1},i_{2},i_{3}}\left(  x^{i_{1}}\otimes t^{i_{2}%
}\right)  \left(  f\sum_{\ell=1}^{p-1}\frac{1}{\ell!\left(  p-\ell\right)
!}x^{p^{r}\ell}\otimes t^{p^{r}\left(  p-\ell\right)  }\right)  ^{i_{3}}.
\end{align*}
After expanding, the tensors are of the form $x^{i_{1}+p^{r}\ell^{\prime}%
}\otimes t^{i_{2}+p^{r}\ell^{\prime\prime}},$ $\ell^{\prime},\ell
^{\prime\prime}$ as before. Applying $1\otimes z_{p^{s}}$ to this expression
will give $0$ unless%
\begin{equation}
p^{s}=i_{2}+p^{r}\ell^{\prime\prime}. \label{spow}%
\end{equation}
Assume first that $s<r.$ Since $p^{r}>p^{s}$ we see that $i_{3}=0$ and
$i_{2}=p^{s}.$ Thus $i_{1}=i-p^{s}$ and we get%
\begin{align*}
z_{p^{s}}\left(  x^{i}\right)   &  =\binom{i}{i-p^{s},p^{s},0}x^{i-p^{s}%
}z_{p^{s}}\left(  t^{p^{s}}\right) \\
&  =\binom{i}{p^{s}}x^{i-p^{s}}\\
&  =i_{\left(  s\right)  }x^{i-p^{s}},
\end{align*}
as desired.

Now we consider the case $s=r.$ Then $i_{3}=0,\;i_{2}=p^{r},\;i_{1}=i-p^{r}$
certainly satisfies eq. $\left(  \ref{spow}\right)  $. However, we get an
additional solution to this equation, namely $i_{3}=1,\;\ell=p-1,\;i_{2}%
=0,\;i_{1}=i-1.$ Thus%
\begin{align*}
z_{p^{r}}\left(  x^{i}\right)   &  =\binom{i}{i-p^{r},p^{r},0}x^{i-p^{r}%
}z_{p^{r}}\left(  t^{p^{r}}\right)  +\binom{i}{i-1,0,1}x^{i-1}f\frac{1}%
{\left(  p-1\right)  !\left(  p-\left(  p-1\right)  \right)  !}x^{p^{r}\left(
p-1\right)  }z_{p^{r}}\left(  t^{p^{r}}\right) \\
&  =i_{\left(  r\right)  }x^{i-p^{r}}-ifx^{p^{r}\left(  p-1\right)  +i-1}.
\end{align*}
\end{proof}

The results above do not generalize easily to the case $n>r+1.$ Certainly, if
$2r<n$ then the comultiplication on $H_{n,r,f}$ (and its coaction on $L$)
becomes much more complicated, making the computations of the algebra
structure (and the action) of its dual much more involved as well. If
$r+1<n\leq2r,$ computation of the algebra structure of $H_{n,r,f}^{\ast}$ is
somewhat more complex than the case considered here -- in particular,
$z_{p^{r}}^{p}\neq fz_{1}$-- but, as a future paper \cite{Koch14b} will show,
it is possible to show that $H_{n,r,f}^{\ast}$ is generated as a $K$-module by
$\left\{  \prod_{s=0}^{n-1}z_{p^{s}}^{j_{s}}:0\leq j_{s}\leq p-1\right\}  ,$
and much of its action on $L$ can be made explicit.

\section{A Note\ on Modular Extensions}

While the focus of this work is primitive purely inseparable extensions, it
should be pointed out that the constructions here can be adapted easily to
general modular extensions. The following should be clear.

\begin{proposition}
Let $L/K$ be modular, $L\cong L_{1}\otimes\cdots\otimes L_{s}$ with $L_{i}/K$
primitive of degree $p^{n_{i}},\;1\leq i\leq s.\;$For each $i$, pick
$0<r_{i}<n_{i}$ and $f_{i}\in K^{\times}$. Set%
\[
H=H_{n_{1},r_{1},f_{1}}\otimes\cdots\otimes H_{n_{s},r_{s},f_{s}}.
\]
Then $L$ is an $H$-Galois object.
\end{proposition}

Thus, the constructions in the previous sections show that any modular
extension can be equipped with numerous Hopf Galois structures. However, it is
not the case that all (local-local) Hopf Galois structures on modular
extensions have been exhibited here, as the following example shows.

\begin{example}
Let $K=\mathbb{F}_{p}\left(  T_{1},T_{2}\right)  .$ Let $L=K\left(
x,y\right)  $ with $x^{p^{2}}=T_{1},y^{p^{2}}=T_{2}.\;$Then $L/K$ is modular.
Let $H=K\left(  t,u\right)  /\left(  t^{p^{2}},u^{p^{2}}\right)  ,$ and define
$\Delta:H\rightarrow H\otimes H$ by%
\begin{align*}
\Delta\left(  t\right)   &  =S_{1}\left(  \left(  u^{p}\otimes1,t\otimes
1\right)  ;\left(  1\otimes u^{p},1\otimes t\right)  \right)  \\
\Delta\left(  u\right)   &  =S_{0}\left(  u\otimes1,1\otimes u\right)  .
\end{align*}
Along with the counit $\varepsilon$ given by $\varepsilon\left(  t\right)
=\varepsilon\left(  u\right)  =0$ and antipode $\lambda\left(  t\right)
=-t,\lambda\left(  u\right)  =-u,$ this gives $H$ the structure of a $K$-Hopf
algebra which is not monogenic. While checking that the Hopf algebra axioms
hold is straightforward, one can also note that $H\otimes K^{p^{-\infty}}$ is
a special case of a ``bigenic'' Hopf algebra given in \cite[Ex. 6.6]{Koch12b}.
Define $\alpha:L\rightarrow L\otimes H$ by%
\begin{align*}
\alpha\left(  t\right)   &  =S_{1}\left(  \left(  y^{p}\otimes1,x\otimes
1\right)  ;\left(  1\otimes u^{p},1\otimes t\right)  \right)  \\
\alpha\left(  u\right)   &  =S_{0}\left(  y\otimes1,1\otimes u\right)  .
\end{align*}
Then $L$ is an $H$-Galois object.
\end{example}

In fact, this example seems to suggest that the biggest obstacle to a modular
extension $L$ being an $H$-Galois object is the algebra structure of $H$; the
coalgebra structure seems to naturally give a coaction. If $L$ is an
$H$-Galois object, then $L\otimes L\cong L\otimes H$.\ Also, $L\otimes L$ is a
truncated polynomial algebra: in the simple, degree $p^{n}\;$case $L=K\left(
x\right)  $, clearly we have $L\left(  u\right)  /\left(  u^{p^{n}}\right)
\cong L\otimes L\ $via $u\mapsto1\otimes x-x\otimes1$. Thus, for a given
modular extension $L/K,$ the problem appears to be reduced to finding the Hopf
algebra structures on the truncated polynomial algebra $L\otimes L.$ We have
not found a local-local Hopf algebra of the proper type (as an algebra) which
does not give $L$ the structure of an $H$-Galois object.

\bibliographystyle{amsalpha}
\bibliography{MyRefs}
\end{document}